\documentclass[12pt]{amsart}
\usepackage{mathrsfs}
\usepackage{url}
\usepackage{color}
\usepackage[a4paper,asymmetric]{geometry}
\usepackage{latexsym}
\usepackage{amsthm}
\usepackage{amssymb}
\usepackage{amsfonts}
\usepackage{amsmath}

\newtheorem{theorem}{Theorem}[section]
\newtheorem{thm}[theorem]{Theorem}
\newtheorem{lemma}[theorem]{Lemma}
\newtheorem{lem}[theorem]{Lemma}
\newtheorem{proposition}[theorem]{Proposition}

\newtheorem{corollary}[theorem]{Corollary}

\theoremstyle{definition}
\newtheorem{definition}[theorem]{Definition}
\newtheorem{defn}[theorem]{Definition}

\theoremstyle{remark}

\newtheorem{rem}[theorem]{Remark}
\numberwithin{equation}{section}

 \DeclareMathAlphabet{\mathpzc}{OT1}{pzc}{m}{it}

 \newcommand{\pp}{\mathbf{P}}

 \newcommand{\E}{\mathbb{E}}            % expectation
 \newcommand{\e}{\varepsilon}

 \newcommand{\N}{\mathbb{N}}
 \newcommand{\R}{\mathbb{R}}

 \newcommand{\FF}{\mathcal{F}}
 \newcommand{\PP}{\mathbb{P}}

 \newcommand{\Be}{\begin{equation}}
 \newcommand{\Ee}{\end{equation}}
 \newcommand{\Bs}{\begin{split}}
 \newcommand{\Es}{\end{split}}
  \newcommand{\Bes}{\begin{equation*}}
 \newcommand{\Ees}{\end{equation*}}
 \newcommand{\BT}{\begin{thm}}
 \newcommand{\ET}{\end{thm}}
 \newcommand{\Bp}{\begin{proof}}
 \newcommand{\Ep}{\end{proof}}
 \newcommand{\BL}{\begin{lem}}
 \newcommand{\EL}{\end{lem}}
 \newcommand{\BP}{\begin{proposition}}
 \newcommand{\EP}{\end{proposition}}
 \newcommand{\BC}{\begin{corollary}}
 \newcommand{\EC}{\end{corollary}}
 \newcommand{\BR}{\begin{rem}}
 \newcommand{\ER}{\end{rem}}
 \newcommand{\BD}{\begin{defn}}
 \newcommand{\ED}{\end{defn}}
 \newcommand{\BI}{\begin{itemize}}
 \newcommand{\EI}{\end{itemize}}

  \newcommand{\dif}{{\rm d}}

\def\EE{\mathbb{E}}
\def\LL{\mathcal L}\def\PP{\mathbb P}

\def\<{\left<}\def\>{\right>}
\def\({\left(}\def\){\right)}

\begin{document}
\title
[Asymptotics of the entropy production rate for  OU Process]{Asymptotics of the entropy production rate for $d$-dimensional Ornstein-Uhlenbeck processes}

\author[R. Wang]{Ran Wang}
\address{School of Mathematical Sciences, University of Science and Technology of China, Hefei, China.} \email{wangran@ustc.edu.cn}

\author[L. Xu]{Lihu Xu}
\address{Department of Mathematics, Faculty of Science and Technology,  University of  Macau, Taipa, Macau.}
\email{lihuxu@umac.mo}

\maketitle
\begin{minipage}{140mm}
\begin{center}
{\bf Abstract}
\end{center}
\ \ \ \ In  the context of  non-equilibrium statistical physics, the entropy production rate is an important concept to describe how far a specific state of a system is from its equilibrium state.
In this paper, we establish a central limit theorem and  a moderate deviation principle for  the entropy production rate of $d$-dimensional Ornstein-Uhlenbeck processes,  by the techniques of functional inequalities such as Poincar\'e inequality and log-Sobolev inequality.  As an application, we obtain a law of  iterated logarithm for the entropy production rate.

\end{minipage}

\vspace{4mm}

\medskip
\noindent
{\bf Keywords}: Entropy production rate; Ornstein-Uhlenbeck process; Central limit theorem; Moderate deviation principle; Law of  iterated logarithm.
\medskip
\noindent
%{\bf Keywords}: Stochastic real Ginzburg-Landau equation; $\alpha$-stable noises;  Large deviation principle; Strong Feller property; Irreducibility; Occupation measure.

\medskip
\noindent
{\bf Mathematics Subject Classification (2010)}: \ {60F05; 60F10; 60G10}.

%%% ----------------------------------------------------------------------

\section{Introduction}

   Entropy production rate (EPR) was first proposed in non-equilibrium statistical physics to describe how far a specific state of a system
    is from its equilibrium state (\cite{JNPS14,JQQ,LS,Ma,TMO12}), which plays a key role in thermodynamics
of irreversible processes.  EPR is usually defined by the relative entropy of the process with respect to its time reversal.

%Qian {\it et al.} gave a unified measure-theoretic definition of entropy production rate for stochastic %processes: it is defined as the specific relative entropy of the process with respect to its time reversal. %The Markov process is reversible if and only if its entropy production rate vanishes.
 %   See the book \cite{JQQ} and references therein.

%There exists a great amount of literature on the entropy production rate.

For stationary Markov processes, by ergodic theorem, the sample entropy production rate (see \eqref{epr})  converges to the EPR almost surely. Furthermore, in many cases, the sample EPR has a large deviation property and the rate function has a symmetry of Gallavotti-Cohen type,  see  \cite{GJ, JQZ,LS} and so on. Also see  survey papers \cite{Ge} and \cite{Tou} for  applications of large deviations in non-equilibrium systems.

\vskip0.3cm

However, there seem no other asymptotic results such as moderate deviation principle (MDP), central limit theorem (CLT) and law of iterated logarithm (LIL) for EPR.
For the general high dimensional diffusion processes, it seems very hard to establish these asymptotics. The difficulties arise
from the several points as follows.  On the one hand, the sample EPR has a complicated form which includes the density function of the invariant probability measure (see  \eqref{sample epr} below). Usually
the density function does not have an explicit expression or it is too complicated to use.
On the other hand, due to the complicated form of sample EPR again, there are not known criteria
directly applicable to prove MDP.

\vskip0.3cm

The aim of this paper is to partly fill in this gap by studying the above asymptotics for EPR of $d$-dimensional Ornstein-Uhlenbeck processes. Thanks to the nice form of the density function of invariant measure, we use the techniques of functional inequalities such as Poincar\'e inequality and log-Sobolev inequality, to establish a CLT and an MDP.  As an application of MDP, we obtain an LIL for EPR.
To show LIL, we first sample a Markov chain from the process and prove its LIL, then fill in the gap between the chain. Due to the lack of independence, we sample the Markov chain after a suitable long time
to get some weak dependence when proving LIL.

\vskip 3 mm

The paper is organized as follows. In Section 2, we first introduce the EPR for Ornstein-Uhlenbeck process, and then give the main results of this paper. In Section 3,  we recall a criterion of CLT and MDP for additive functionals of   Markov process, and then prove
the first main result. In the last section, we prove LIL by the technique of MDP.

\section{Ornstein-Uhlenbeck Processes and Entropy Production Rate}
We are concerned with the Ornstein-Uhlenbeck (OU) process $\xi=\{\xi_t\}_{t\ge0}$ which is the solution of the following stochastic differential equation
\begin{equation}\label{eq: OU}
\dif \xi_t=B\xi_t\dif t +\Sigma\dif W_t, \ \ \ \
\end{equation}
with the initial value $\xi_0$, where $B=(b_{ij})\in \R^{d\times d}$  and  $\Sigma=(\sigma_{ij})\in \R^{d\times d}$ are both constant  matrices, and $\{W_t\}_{t\ge0}$ is a $d$-dimensional Brownian motion on a probability space $(\Omega, \FF, \mathbb P)$ with a filtration $\{\FF_t\}_{t\ge0}$.  We assume that
\begin{equation}\label{A1}
   \text{ the diffusion coefficient $Q:=\Sigma\Sigma^T$ is strictly positive definite.}
\end{equation}
It is well known that the Markov process $\xi$ has an invariant probability $\mu$ if and only if
\begin{equation}\label{A2}
  \ \ \  \text{all the eigenvalues of $B$ have negative real parts,}
\end{equation}
 see \cite{CG02,Qian01}, and in that case
\begin{align}\label{mu}
\mu=N(0, \Gamma),
\end{align}
the Gaussian distribution on $\R^d$ with mean vector $0$ and covariance matrix
% which is the unique symmetric solution of the equation
%$$
%C+B\Gamma+\Gamma B^T=0.
%$$
%$\Gamma$ can be expressed explicitly as
\begin{align}\label{gamma}
\Gamma=\int_0^{\infty}e^{sB}Qe^{sB^T}\dif s.
\end{align}
Notice that $\Gamma$ is nonsingular under conditions \eqref{A1} and \eqref{A2}. Let $\rho(x)$ be the probability density function of  $\mu$ with respect to the Lebesgue measure  $\dif x$ on $\R^d$. Then
$$
\rho(x)=(2\pi)^{-\frac d2}(\det \Gamma)^{-\frac12}\exp\left\{-\frac{\langle x, \Gamma^{-1}x\rangle}{2} \right\}.
$$
We refer to \cite{JQQ} for more details.

\vskip0.3cm

In this paper, we always assume that \eqref{A1} and \eqref{A2} hold, and the initial measure of $\xi$ is the invariant measure $\mu$. Thus, $\xi$ is stationary. % with the invariant initial distribution. %Let $\bar\Omega$ be the trajectory space $C([0,\infty);\R^d)$, and $\bar\FF^s_t$  the Borel $\sigma$-algebra generated by $\{\omega_u,s\le u\le t\}$ for any $0\le s\le t$.
Let $\pp_{[s,t]}$ and $\pp^{-}_{[s,t]}$ be  the distributions of $\{\xi_u\}_{s\le u\le t}$ and  $\{\xi_{t+s-u}\}_{s\le u\le t}$ for any $0\le s\le t$.
Recall that the process $\xi=\{\xi_t\}_{t\ge0}$ is said to be {\bf reversible} if   $\pp_{[s,t]}=\pp^{-}_{[s,t]}$  for any $0\le s\le t$. It is shown that the stationary OU process $\xi$ is reversible if and only if the coefficients $B$ and $Q$ satisfy the symmetry condition $Q^{-1}B=(Q^{-1}B)^T$, see \cite{CG02, Qian01}. Notice that when $d=1, B<0$, the OU process $\xi$ is always reversible.

\vskip0.3cm
Now we give the definition and properties of the EPR of $\xi$, see \cite{JQQ} for details.

\begin{definition} For  the stationary diffusion process $\xi$,   the {\bf  sample entropy production rate} is defined by
\begin{align}\label{epr}
  e_p(t)(\omega):=\frac1t\log\frac{\dif \pp_{[0,t]} }{\dif \pp^-_{[0,t]}}(\omega),
\end{align}
and the  {\bf  entropy production rate} is defined by
\begin{align}
e_p:= \lim_{t\rightarrow+\infty} \E^{\pp_{[0,t]}} \left[ e_p(t)  \right].
\end{align}

\end{definition}

 By the ergodic theorem of the stationary process,   for $\mathbb P$-almost every $\omega\in \Omega$,
 $$
 \lim_{t\rightarrow \infty}e_p(t)(\omega)=e_p.
 $$
By the well-known Cameron-Martin-Girsanov formula, one can prove that the two probability measures  $\pp_{[0,t]}$ and $\pp^-_{[0,t]}$ on $(\Omega,\FF_t)$ are equivalent. Moreover, for $\mathbb P$-almost every $\omega\in\Omega$,  the Radon-Nikodym derivative
\begin{align}\label{sample epr}
&\frac{\dif \pp_{[0,t]}}{\dif \pp^-_{[0,t]}}(\xi_{\cdot}(\omega) )\\\notag
=&\exp\Big[\int_0^t(2Q^{-1}B\xi_s(\omega)-\nabla\log \rho(\xi_s)(\omega))^T \Sigma\dif W_s(\omega) \\ \notag
 &+\frac{1}{2}\int_0^t((2Q^{-1}B\xi_s(\omega)-\nabla\log \rho(\xi_s)(\omega))^T Q(2Q^{-1}B\xi_s(\omega)-\nabla\log \rho(\xi_s)(\omega))\dif s \Big],
\end{align}
 and
 \begin{align} \label{eq: epr finite}
 e_p=\frac12\int_{\R^d}(2Q^{-1}Bx-\nabla\log \rho(x))^TQ(2Q^{-1}Bx-\nabla\log \rho(x))\rho(x)\dif x<+\infty.
\end{align}
%the entropy production rate $ $ can be expressed as the above quantity. %Hence, the process $\xi$ is reversible if and only if the entropy production rate $e_p$ vanishes, or equivalently,

\vskip0.3cm
In this paper, we shall investigate deviations of $e_p(t) $ from the EPR $e_p$, as $t$ increases to $+\infty$,
that is,  the asymptotic behavior of
$$
\frac{\sqrt t }{ \lambda(t)}\left(e_p(t) -e_p\right), \ \ \ t\ge0,
$$
 where $\lambda(t)$ is some deviation scale which strongly influences the above asymptotic behavior:
\begin{itemize}
  \item[(1)]if $\lambda(t)= 1$, we are in the domain of the central limit theorem (CLT);
  \item[(2)] the case $\lambda(t)= \sqrt t$ provides some large deviation principle (LDP) estimates;
  \item[(3)] to fill in the gap between the CLT scale  and the large deviations scale, we shall study the so-called moderate deviation principle (MDP), see \cite{Dembo-Zeitouni}, that is when the deviation scale satisfies
\begin{equation} \label{lambda}
 \lambda(t)\to+ \infty,\ \   \lambda(t)/\sqrt t\to0\ \ \ \text{as}\ t\to+\infty.
 \end{equation}\end{itemize}

\vskip0.3cm

For any $p>1, l>0$, let
\Be \label{e:Upl}
U^p_{\mu}(l):=\left\{\nu\in \mathcal M_1(\R^d);\nu\ll\mu \text{ and }   \left\|\frac{\dif \nu}{\dif \mu}\right\|_{L^p(\mu)}\le l  \right  \},
\Ee
here $\mathcal M_1(\R^d)$ is the space of all probability measures on $\R^d$.

The first main result is the following theorem about CLT and MDP of EPR.

\begin{theorem} \label{thm main}
Under conditions \eqref{A1} and \eqref{A2},     $e_p(t)$ has the following properties:
\begin{itemize}
  \item[(1)]$ \mathbb P_{\nu}(\sqrt t[e_p(t)-e_p]\in \cdot, )$ converges weakly to  a centered Gaussian measure $N(0, \sigma^2)$ with variance $\sigma^2$ as $t\rightarrow \infty$, uniformly  for the initial distributions $\nu \in U_{\mu}^p(l)$ for any $p>1, l>0$;
  \item[(2)]  as $t\rightarrow +\infty$,  $\mathbb P_{\nu} \left(\frac{\sqrt t}{\lambda(t)}[e_p(t)-e_p]\in \cdot\right)$ satisfies a large deviation principle with speed $\lambda^{-2}(t)$ and rate function $I(x)=x^2/(2\sigma^2)$, uniformly  for the initial distributions $\nu \in U_{\mu}^p(l)$ for any $p>1, l>0$,
\end{itemize}
here the asymptotic variance $\sigma^2:=\lim_{t\rightarrow \infty} t \EE\left[e_p(t)-e_p \right]^2$ exists, and  $\lambda(t)$ satisfies \eqref{lambda}.
\end{theorem}
As an applicaiton of the above theorem, we establish the second main result about LIL in the theorem below:
\begin{theorem}  \label{t:LIL}
Under conditions \eqref{A1} and \eqref{A2},  we have
\begin{align}\label{1}
\limsup_{t\rightarrow +\infty}\sqrt{\frac{t}{ 2\log\log t}} \left( e_p(t)-  e_p\right)=\sigma, \ \  a.s.,
\end{align}
and
\begin{align}\label{2}\liminf_{t\rightarrow +\infty}\sqrt{\frac{t}{ 2\log\log t}} \left( e_p(t)-  e_p\right)=-\sigma, \ \  a.s..
\end{align}
\end{theorem}

\section{The proof of Theorem \ref{thm main}}
\subsection{A Criterion of CLT and MDP}

First, recall the definition of LDP and MDP, see \cite{Dembo-Zeitouni} and \cite{Wu95}.
Let $\mathcal X$ be a complete separable metric space, $\mathcal B(\mathcal X)$ the Borel $\sigma$-field of $\mathcal X$, and $\{X_t\}_{t\ge0}$ a family of  stochastic processes valued in $\mathcal X$.

\begin{definition}(1)  $\{X_t\}_{t\ge0}$ satisfies a large deviation principle (LDP) if there exist a family of positive numbers $\{h(t)\}_{t\ge0}$ which tends to $+\infty$ and a function $I(x)$ which maps $\mathcal X$ into $[0,+\infty]$ satisfying the following conditions:
  \begin{itemize}
  \item[(i)]  for each $l<+\infty$, the level set $\{x:I(x)\le l\}$ is compact in $\mathcal X$;
  \item[(ii)] for each closed subset $F$ of $\mathcal{X}$,
              $$
                \limsup_{t\rightarrow \infty}\frac1{h(t)}\log\mathbb P (X_t\in F)\leq- \inf_{x\in F}I(x);
              $$
   \item[(iii)]for each open subset $G$ of $\mathcal{X}$,
              $$
               \liminf_{t\rightarrow \infty}\frac1{h(t)}\log\mathbb P(X_t\in G)\geq- \inf_{x\in G}I(x),
              $$
   \end{itemize}
here $h(t)$ is called the  speed function and $I(x)$ is the rate function.

(2)  $\{X_t\}_{t\ge0}$ satisfies a  moderate deviation principle (MDP), if $\{\frac{\sqrt t}{\lambda (t)}(X_t-\E X_t)\}_{t\ge0}$ satisfies a large deviation principle, for a family of positive numbers $\lambda(t)$ such that
$$ \lambda(t)\to+ \infty,\ \   \lambda(t)/\sqrt t\to0\ \ \ \text{as}\ t\to+\infty.$$
\end{definition}

\vskip0.3cm
Next, we  recall  a  criterion of CLT and MDP for additive functionals of  Markov process in \cite{Wu95}.

Let $(\Omega, \FF,(\FF_t)_{t\ge0}, (X_t)_{t\ge0},(\theta_t)_{t\ge0}, (\mathbb P_x)_{x\in \mathcal X})$ be a Markov process, and $(P_t)_{t\ge0}$ be associated  semigroup of Markov kernels   with $P_0(x,\cdot)=\delta_x$. We further assume that $(X_t)_{t\ge0}$ is continuous and $(P_t)_{t\ge0}$ possesses an invariant measure $\mu$. For simplicity, we write $\mathbb P_{\mu}:=\int \mathbb P_x \mu(\dif x),\ \mu(\cdot)$ is the usual integral in $L^1(\mu)$, and  $\langle \cdot,\cdot\rangle_{\mu}$ the usual inner production in $L^2(\mu)$.

Let $A=(A_t)_{t\ge0}$ be an additive functional (i.e., $ A_{t+s}=A_t+A_s\circ  \theta_t, \forall t,s\ge0$) such that $A_t$ is $\FF_t$-measurable, $\forall t\ge0$. For any measurable function $g$, let
\begin{equation}
P_t^Ag(x):=\E^x g(X_t)e^{A_t}.
\end{equation}

\begin{theorem}\cite[Theorem $2.4'$]{Wu95}\label{thm wu 95}
Assume that there exist  $t_i>0, i=1,2,3$, satisfying that
\begin{itemize}
  \item[(1)] $1$ is an isolated, simple and the only eigenvalue with modulus $1$ of $P_{t_1}$  in $L^2(\mu)$;
  \item[(2)]  $P_{t_2}$ is hyperbounded, i.e.,
  \begin{equation*}\label{eq: condition 2}
      \exists q>p>1 \text{ so that } P_{t_2}: L^p(\mu)\rightarrow L^q(\mu) \  \ \text{is bounded};
  \end{equation*}
  \item[(3)]  there exists $\delta>0$  such that
  \begin{equation*}\label{eq: condition 3}
  \sup_{0\le s\le t_3}\E_{\mu}\exp(\delta|A_s|)<\infty.
  \end{equation*}
\end{itemize}
Then  the following  limit exists
 \begin{equation*}
V(A):=\lim_{t\rightarrow \infty}\frac1t\left[\E_{\mu}A^2_t-(\E_{\mu}A_t)^2\right]<+\infty,
\end{equation*}
 and as $t\rightarrow +\infty$,
\begin{itemize}
  \item[(i)] $\mathbb P_{\nu}\left(\frac1{\sqrt t}[A_t-\E_{\mu}(A_t)]\in \cdot\right)$ converges weakly to a centered Gaussian measure $N(0, V(A))$ with variance $V(A)$,  uniformly  for the initial distributions $\nu \in U_{\mu}^p(l)$ for any $p>1, l>0$,
  \item[(ii)] $\mathbb P_{\nu}\left(\frac1{\sqrt t\lambda(t)}[A_t-\E_{\mu}(A_t)]\in \cdot\right)$ satisfies an LDP with speed $\lambda^{2}(t)$ and rate function $I(x)=x^2/(2V(A))$,  uniformly  for the initial distributions $\nu \in U_{\mu}^p(l)$ for any $p>1, l>0$,
\end{itemize}
where
$U^p_{\mu}(l)$ is defined in \eqref{e:Upl}
and $\lambda(t)$ satisfies \eqref{lambda}.

\end{theorem}

\subsection{The proof for  Theorem \ref{thm main}}
  To prove Theorem \ref{thm main}, it suffices to verify the conditions of Theorem \ref{thm wu 95} for the stationary OU semigroup $(P_t)_{t\ge0}$ and for the additive functional $A_t:=t\cdot e_p(t)$, which is
 \begin{align}\label{eq: A}
A_t=&\int_0^t(2Q^{-1}B\xi_s -\nabla\log \rho(\xi_s)) ^T \Sigma\dif W_s   \\ \notag
& +\frac{1}{2}\int_0^t (2Q^{-1}B\xi_s -\nabla\log \rho(\xi_s) )^TQ(2Q^{-1}B\xi_s -\nabla\log \rho(\xi_s)) \dif s .
\end{align}
   These will be done in this  section by the techniques of functional inequalities, such as Poincar\'e inequality and log-Sobolev inequality, see \cite{BGL, Wang}.
\vskip0.3cm

It is well-known that  the  centered  Gaussian measure  $\mu$ on $\R^d$ with covariance matrix $\Gamma$ satisfies  {\bf Poincar\'e  inequality}
\begin{align}\label{eq: Poincare}
\int f^2\dif \mu-\left(\int f\dif \mu\right)^2&\le \int\Gamma\nabla f\cdot\nabla f\dif\mu,\\ \notag
%&\le \lambda_{max}(\Gamma)\int_{\R^d}|\nabla f|^2\dif \mu
\end{align}
and  {\bf log-Sobolev inequality}
\begin{equation}\label{eq: LSI 2}
\int f^2\log f^2\dif \mu- \int f^2\dif \mu\log\int f^2\dif \mu\le 2\int \Gamma\nabla f\cdot\nabla f\dif \mu,
\end{equation}
for every smooth function $f:\R^d\rightarrow \R$, see  \cite[Page 181, Page 258]{BGL}.

\vskip0.3cm
The generator of OU semigroup $(P_t)_{t\ge0}$ is
$$
\LL=\frac12\nabla\cdot Q\nabla+Bx\cdot\nabla\ \ \text{ with dominant } \mathcal D(\LL).
$$
For any given $t_0>0$, the time reversal $\xi^-=\{\xi_{t_0-t}\}_{0\le t\le t_0}$ of the stationary diffusion $\xi$ over the time interval $[0,t_0]$ is also stationary. Its infinitesimal generator is
$$
\LL^-=\frac12\nabla\cdot Q\nabla+(-Bx+Q\nabla \log\rho)\cdot \nabla\ \ \text{ with dominant } \mathcal D(\LL^-).
$$
See \cite[Theorem 3.3.5]{JQQ}.

When $d$-dimensional stationary OU process $\xi$ is  not reversible, the generator $\LL$ is not self-adjoint.
For $f, g\in \mathcal D (\LL)\cap L^{2}(\mu)$, consider the symmetrized Dirichlet form
$$
\mathcal E(f,g):=\frac12[\langle -\LL f, g\rangle_{\mu}+ \langle -\LL^- f, g\rangle_{\mu}].
$$
Using the formula of integration by parts,  we obtain that
\begin{align}\label{dirichlet}
\langle -\LL f, f\rangle_{\mu}=\mathcal E(f,f) =&\frac12\langle Q\nabla f,\nabla f\rangle_{\mu}\\\notag
 \ge&\frac{\lambda_{\min}(Q)}{2\lambda_{\max}(\Gamma)}\langle \Gamma\nabla f,\nabla f\rangle_{\mu},
\end{align}
  where $\lambda_{\min}(Q)$ is the minimum eigenvalue of $Q$, $\lambda_{\max}(\Gamma)$ is the maximum eigenvalue of $\Gamma$ (see \eqref{gamma}).

\vskip0.3cm

Now we verify the conditions in Theorem \ref{thm wu 95}.
\subsubsection{{\bf Condition (1).} }  Putting \eqref{eq: Poincare} and \eqref{dirichlet} together, we have
$$
\int f^2\dif \mu-\left(\int f\dif \mu\right)^2\le \frac{2\lambda_{\max}(\Gamma)}{\lambda_{\min}(Q)}\langle -\LL f, f\rangle_{\mu}.
$$
By \cite[Section 1.1]{Wang}, this is equivalent to that: for every $f:\R^d\rightarrow \R$ in $L^2(\mu)$ with $\mu(f)=0$,
$$
\int (P_tf)^2\dif \mu\le    \exp\left(-\frac{ \lambda_{\min}(Q)}{\lambda_{\max}(\Gamma)}t\right)\int f^2\dif \mu,
$$
 which implies that condition $(1)$ holds for any $t>0$.

\vskip0.5cm
\subsubsection{{\bf Condition (2).}} For reversible Markov processes,  by the celebrated Gross theory,  the log-Sobolev inequality is equivalent to the hypercontractivity of the semigroup $(P_t)_{t\ge0}$, which is stronger than hyperboundness  in Condition 2 of Theorem \ref{thm wu 95}, see \cite[Section 5.2]{BGL}. However,  this equivalent relationship does not always hold in the non-reversible case.

Fortunately,  the $d$-dimensional OU semigroup $(P_t)_{t\ge0}$ also has the the following  hypercontractivity by \cite{Fuh}.

\begin{lemma}\cite[Theorem 2.2, Proposition 4.2]{Fuh}.\label{lem hyperbound} For every $t>0$, there exist $q>p>1$ such that the  OU semigroup $(P_t)_{t\ge0}$ is bounded from $L^p(\mu)$ to $L^q(\mu)$ and
$$
\|P_tf\|_{L^q(\mu)}\le \|f\|_{L^p(\mu)}, \ \ \ \ \forall f\in L^p(\mu).
$$
\end{lemma}

\vskip0.5cm

\subsubsection{{\bf Condition (3).}}
In this step, we prove the following stronger result, which implies Condition (3).
\begin{proposition}\label{lem integ} There exist $\delta>0,c>0$ such that for any $t>0$,
  \begin{equation}\label{eq: Cramer 1}
\frac1t\log\E_{\mu}\exp(\delta|A_t|)\le c.
  \end{equation}
\end{proposition}

\begin{proof}
Let
$$
M_t:=\int_0^t(2Q^{-1}B\xi_s -\nabla\log \rho(\xi_s)) ^T \Sigma\dif W_s.
$$
Then the quadratic variation process  of local martingale $M_t$ is
$$
\langle M \rangle_t= \int_0^t\left[2Q^{-1}B\xi_s -\nabla\log \rho(\xi_s) \right]^T Q\left[2Q^{-1}B\xi_s -\nabla\log \rho(\xi_s)\right] \dif s,
$$
and
$$
A_t=M_t+\frac12\langle M\rangle_t.
$$
For any $s\in[0,t], \delta\in(0,1)$, by H\"older's inequality, we have
\begin{align}\label{eq: 1}
\E_{\mu}[\exp(\delta|A_s|)]&\le \E_{\mu}\left[\exp\left(\delta|M_s|+\frac \delta 2 \langle M\rangle_s\right)\right]\\\notag
&\le \sqrt{ \E_{\mu}\left[\exp\left(2\delta|M_s| \right)\right]\cdot\E_{\mu}\left[\exp\left(  \delta  \langle M\rangle_s\right)\right]}\\\notag
&\le \sqrt{ \E_{\mu}\left[\exp\left(2\delta M_s \right)+\exp\left(-2\delta M_s \right)\right]\cdot\E_{\mu}\left[\exp\left(  \delta  \langle M\rangle_s\right)\right]}.
\end{align}
Now, we assume that
\begin{align}\label{eq: Novikov}
\E_{\mu}\left[\exp\left(  2\delta   \langle M\rangle_t\right)\right]\le e^{ct}, \ \ \text{for some constant } c\in(0,+\infty).
\end{align}
By the Novikov's criterion \cite[Page 332]{RY},  the local martingales  $(\exp\left(\pm2\delta M_s-2\delta^2\langle M\rangle_s \right))_{0\le s\le t}$   are martingales. Then for all $s\in[0,t]$,
$$
\E_{\mu} [\exp\left(\pm2\delta M_s \right)]=\E_{\mu}\left[\exp\left(  2\delta^2  \langle M\rangle_s\right)\right]\le e^{ct}.
$$
  This inequality, together with inequality \eqref{eq: 1}, implies the desired result \eqref{eq: Cramer 1}.
\vskip0.3cm
Now, it remains to prove \eqref{eq: Novikov}.
Denote
$$
\tilde Q=(2Q^{-1}B+\Gamma^{-1})^T Q(2Q^{-1}B+\Gamma^{-1}).
$$
Then
$$
\langle M\rangle_t\le \lambda_{\max}(\tilde Q)\int_0^t|\xi(s)|^2\dif s,
$$
where $\lambda_{\max}(\tilde Q)$ is the maximum eigenvalue of $\tilde Q$. To prove \eqref{eq: Novikov}, it is sufficient to prove that there exists $c>0, \eta>0$ such that
\begin{align}\label{eq: integ}
\E_{\mu}\left[\exp\left(  \eta \int_0^t|\xi(s)|^2\dif s \right)\right]\le e^{ct},\ \ \ \  \forall  t\ge0.
 \end{align}

Next, we shall prove \eqref{eq: integ}. For any $\mu$-integrable function $V:\R^d\rightarrow \R$, consider the Feynman-Kac semigroup
$$
P_t^V f(x)=\E^x f(\xi_t)\cdot\exp\left(\int_0^t V(\xi_s)\dif s\right),
$$
where $f\ge0$ is Borel measurable. Define the norm
$$
\|P_t^V\|_2:=\sup\{\|P_t^Vf\|_{L^2(\mu)}; f\geq0 \text{ and } \langle f^2\rangle_{\mu}\le1 \}.
$$
Taking $f\equiv 1$, by Cauchy-Schwartz inequality,  we  obtain that
\begin{align}\label{eq: holder}
\E_{\mu}\exp\left(\int_0^t V(\xi_s)\dif s \right)\le \|P_t^V\|_2.
\end{align}

 Now, we prove \eqref{eq: integ}, based on the following result from \cite{Wu00}.

\begin{lemma}\cite[Corollary 4]{Wu00}\label{lem wu 00}
Assume the  log-Sobolev inequality holds, i.e., there exists $c>0$ such that for all $f\in \mathcal D (\LL)$,
\begin{equation*}\label{eq: LSI}
\int f^2\log f^2\dif \mu- \int f^2\dif \mu\log\int f^2\dif \mu\le c\langle -\LL f, f \rangle_{\mu}.
\end{equation*}
Then for any $V\in L^1(\mu)$,
$$
\frac1t\log\|P_t^V\|_2\le \frac1c\log\int e^{cV}\dif \mu.
$$
\end{lemma}

 Combining \eqref{eq: LSI 2} and \eqref{dirichlet}, and applying Lemma \ref{lem wu 00} and \eqref{eq: holder} to $V(x)=\eta |x|^2$ for some positive number $\eta$ (to be determined later), we have
 \begin{align*}
 \frac1t\log \E_{\mu}\left[\exp\left(\eta \int_0^t|\xi_s|^2\dif s \right) \right]
 \le &  \frac{\lambda_{\min}(Q)}{4\lambda_{\max}(\Gamma)}\log\int e^{\frac{4\lambda_{\max}(\Gamma)}{\lambda_{\min}(Q)}\eta|x|^2}\mu(\dif x).\\\notag
 %= &  \frac{\lambda_{\min}(Q)}{4\lambda_{\max}(\Gamma)}\log\int e^{\frac{4\lambda_{\max}(\Gamma)}{\lambda_{\min}(Q)}\eta|x|^2} \rho(x)\dif x.\\\notag
  \end{align*}
The above integral is finite, once
% $$
 %\frac{2\eta\lambda_{\max}(\Gamma)}{\lambda_{\min}(Q)}-\frac1{2\lambda_{\max}(C)}>0,
% $$
 %or equivalently
 \begin{align}\label{eq: eta}
 \eta<\frac{\lambda_{\min}(Q)}{8\lambda_{\max}^2(\Gamma)}.
 \end{align}
 Thus, the inequality \eqref{eq: integ} holds  for all $ \eta$ satisfying \eqref{eq: eta}. The proof is complete.
\end{proof}

\section{Proof of Theorem \ref{t:LIL}}

For any $t\ge 0$, we define the process
\begin{align}\label{eq: S}
S_t=t\cdot (e_p(t)-e_p).
\end{align}
As $t\rightarrow +\infty$, by the CLT in Theorem \ref{thm main}, $S_t/\sqrt{2t\log\log t} $ will go to $0$ in probability with respect to $\mathbb P$, but the convergence will not be almost sure. Before proving LIL, we shall establish some auxiliary lemmas below.

\subsection{Auxiliary lemmas}
Recall that the unique solution of \eqref{eq: OU} can be expressed as
\begin{equation}\label{eq: solu}
 \xi_t=e^{t B} \xi_0+\int_0^t e^{(t-s)B}D\dif W_s .
\end{equation}
It is a Markov process on $\R^d$ whose transition semigroup is given by
\begin{equation}\label{eq: transition}
P_tf(x):=\E [f(\xi_t)|\xi_0=x]=\int_{\R^d}f(e^{tB}x+y)d\mu_t(y),
\end{equation}
where $\{\mu_t;t\ge0\}$ is a family of centered Gaussian measures on $\R^d$ with covariance matrix
\begin{align}\label{eq: gamma t}
\Gamma_t=\int_0^{t}e^{sB}Qe^{sB^T}\dif s.
\end{align}
Notice that $\Gamma-\Gamma_t$ is positive definite.

Recall that
\begin{align*}
S_t&= \int_0^t(2Q^{-1}B\xi_s -\nabla\log \rho(\xi_s)) ^T \Sigma\dif W_s   \\ \notag
& +\int_0^t\left[\frac{1}{2}(2Q^{-1}B\xi_s -\nabla\log \rho(\xi_s) )^TQ(2Q^{-1}B\xi_s -\nabla\log \rho(\xi_s))-e_p\right] \dif s .
\end{align*}

\begin{lemma}\label{lem OU} For any $t, s>1, x\in\R^d$, there exists a constant $c>0$ such that
$$
\E[|S_{t+s}-S_s|^2|\xi_s=x]\le  c(1+t^2+t|x|^4).
$$

\end{lemma}

\begin{proof}
By the Markov property of $\xi$, it suffices to prove the lemma for $s=0$.
By \eqref{eq: transition}, as $\xi_0=x$, the law of  $\xi_t$  is the  Gaussian measure on $\R^d$ with mean $e^{tB}x$ and  covariance matrix $\Gamma_t$.  Then the next  inequalities follow immediately from the properties of the multivariate Gaussian distribution,
$$
\E\left[|\xi_t|^2|\xi_0=x \right]\le e^{2\lambda_0(B)t}|x|^2+d\lambda_{\max}(\Gamma),
$$
and
$$
\E\left[|\xi_t|^4|\xi_0=x \right]\le c\left( e^{4\lambda_0(B)t}|x|^4+d^2\lambda^2_{\max}(\Gamma)\right),
$$
where $\lambda_0(B)=\max\{\text{the real part of all eigenvalues of } B\}<0$ by \eqref{A2}, $\lambda_{\max}(\Gamma)$ is the maximum principle eigenvalue of $\Gamma$.
By the above inequalities and  H\"older's inequality,  there exists a constant $c>0$ such that
\begin{align*}
&\E[S_t^2|\xi_0=x]\\ \notag
\le& 2 \E\left[ \left(\int_0^t(2Q^{-1}B\xi_s -\nabla\log \rho(\xi_s)) ^T \Sigma\dif W_s\right)^2\bigg|\xi_0=x\right]  \\ \notag
+&2\E \left\{\left[\int_0^t\left(\frac{1}{2}(2Q^{-1}B\xi_s -\nabla\log \rho(\xi_s) )^TQ(2Q^{-1}B\xi_s -\nabla\log \rho(\xi_s))-e_p\right) \dif s\right]^2\bigg|\xi_0=x\right\}\\ \notag
\le &2 \E\left[ \int_0^t(2Q^{-1}B\xi_s -\nabla\log \rho(\xi_s)) ^TQ(2Q^{-1}B\xi_s -\nabla\log \rho(\xi_s))\dif s|\xi_0=x\right]\\ \notag
+& 2t\E\left[\int_0^t\left(\frac{1}{2}(2Q^{-1}B\xi_s -\nabla\log \rho(\xi_s) )^TQ(2Q^{-1}B\xi_s -\nabla\log \rho(\xi_s))-e_p\right)^2 \dif s|\xi_0=x\right]\\
\le &c \E\left[\int_0^t (1+|\xi_s|^2)  \dif s|\xi_0=x\right]+ct\E\left[\int_0^t (1 +|\xi_s|^4)\dif s|\xi_0=x\right]\\ \notag
\le & c(1+t^2+t|x|^4).
 \end{align*}
\end{proof}

\begin{lemma}\label{lem S} For any $p\ge2$, there exists a constant $C_p\in(0,+\infty)$ such that
$$
\E\left[\sup_{n\le t\le n+1}|S_t-S_n|^p\right]\le C_p, \ \ \ \forall n\in\N.
$$
\end{lemma}

\begin{proof}By the Markov property of process $\xi$, we only need to prove this lemma for $n=0$. For any $p\ge2,t\in(0,1)$, there exists a constant $c_p>0$ such that
\begin{align*}
|S_t|^p&\le c_p\left| \int_0^t(2Q^{-1}B\xi_s -\nabla\log \rho(\xi_s)) ^T \Sigma\dif W_s\right|^p   \\ \notag
& +c_p\left|\int_0^t\left[\frac{1}{2}(2Q^{-1}B\xi_s -\nabla\log \rho(\xi_s) )^TQ(2Q^{-1}B\xi_s -\nabla\log \rho(\xi_s))-e_p\right] \dif s \right|^p.
\end{align*}
Taking the supremum up to time $1$ and the expectation, by Burkholder-Davis-Gundy's inequality \cite[Theorem 5.2.4]{DPZ} and H\"older's inequality,  we have

\begin{align*}
&\E\left(\sup_{0\le t\le 1}|S_t|^p\right)\\
\le& 2^{p-1}\E\left(\sup_{0\le t\le 1}\left| \int_0^t(2Q^{-1}B\xi_s -\nabla\log \rho(\xi_s)) ^T \Sigma\dif W_s\right|^p\right)\\
 &+2^{p-1}\E\left(\int_0^1\left|\frac{1}{2}(2Q^{-1}B\xi_s -\nabla\log \rho(\xi_s) )^TQ(2Q^{-1}B\xi_s -\nabla\log \rho(\xi_s))-e_p\right| \dif s \right)^p\\ \notag
\le& c_p\E\left(\int_0^1 (2Q^{-1}B\xi_s -\nabla\log \rho(\xi_s) )^TQ(2Q^{-1}B\xi_s -\nabla\log \rho(\xi_s)) \dif s \right)^{\frac p2}\\ \notag
&+2^{p-1}\E\left(\int_0^1\left|\frac{1}{2}(2Q^{-1}B\xi_s -\nabla\log \rho(\xi_s) )^TQ(2Q^{-1}B\xi_s -\nabla\log \rho(\xi_s))-e_p\right|^p \dif s \right)\\ \notag
\le& c_p \int_0^1\E\left| (2Q^{-1}B\xi_s -\nabla\log \rho(\xi_s) )^TQ(2Q^{-1}B\xi_s -\nabla\log \rho(\xi_s))\right|^{\frac p2} \dif s  \\ \notag
&+2^{p-1} \int_0^1\E\left|\frac{1}{2}(2Q^{-1}B\xi_s -\nabla\log \rho(\xi_s) )^TQ(2Q^{-1}B\xi_s -\nabla\log \rho(\xi_s))-e_p\right|^p \dif s. \\ \notag
 \end{align*}
Since  the $p$-moments  of  Gaussian random variables $\{\xi_s\}_{s\in[0,1]}$ are uniformly  finite, the most right hand side of above inequality  is finite.
\end{proof}

\subsection{Proof of  Theorem \ref{t:LIL}}The proof is  largely inspirited  by the classical proof of LIL for i.i.d. random variable sequence, see \cite[Sectoin 7.5]{Chung}, and some key estimates are heavily based on the MDP results in Theorem \ref{thm main}.

\begin{proof}[Proof of  Theorem \ref{t:LIL}] We shall only prove \eqref{1}, and the proof of \eqref{2} is very similar, which is  omitted here.   The proof is divided into two steps: upper bound and lower bound. Both of steps are heavily based on the MDP results in Theorem \ref{thm main}.

\vskip0.3cm
Step 1. ({\bf Upper bound}). First, we prove the upper bound
\begin{align}\label{upper}
\limsup_{t\rightarrow +\infty}\sqrt{\frac{t}{ 2\log\log t}} \left( e_p(t)-  e_p\right)\le \sigma, \ \  a.s..
\end{align}

  {\it Step 1.1. Upper bound  of the sequence $S_{n_k}, n_k=[\gamma^k], k\ge1$ for any given $\gamma>1$.}  Sample $S_t$ along a discrete sequence $n_k=[\gamma^k]$ for any given $\gamma>1$, and prove the upper bound  for this sequence in this step.

    For any $\theta>0$, let
\begin{align}\label{E}
E_{n,\theta}:=\left\{S_n\ge \theta \sigma \sqrt{2n\log\log n}\right\}.
\end{align}
    Applying Theorem \ref{thm main} to $\lambda(n)=\log\log n$, we have that for any $\theta >1,0<\e<\theta-1$, for any $k$ large enough,
\begin{align*}
\mathbb P\left(E_{n_k,\theta}\right)\le \exp\{- (\theta^2-\e)\log\log n_k\}.
\end{align*}
Consequently, as $\log\log n_k\sim \log k$,  we have
 \begin{align}\label{eq: finite}
\sum_{k}\mathbb P\left(E_{n_k,\theta}\right)<+\infty.
\end{align}
By the Borel-Cantelli lemma and the arbitrariness of $\theta>1$, we obtain
\begin{align}
\limsup_{k\rightarrow +\infty}\frac{S_{n_k}}{\sqrt{2n_k\log\log n_k}}\le \sigma\ \ \ a.s..
\end{align}

\vskip0.3cm

   {\it Step 1.2.  Comparison between the sequence $\{S_n\}_{n\ge1}$   and the  sequence $\{S_{n_k}\}_{k\ge1}$}.  We need to estimate the probability $ \mathbb P(\bigcup_{j=n_{k}+1}^{n_{k+1}}E_{j,\theta\gamma})$. For any integer $j\in (n_k, n_{k+1}]$, define
\begin{align}
F_j:=\left\{ |S_{n_{k+1}}-S_j|\le   \sqrt{ n_{k+1}\log\log k}\sigma \right\}.
\end{align}

Then, for any $k$ large enough,
\begin{align}\label{eq: I0}
E_{n_{k+1}, \theta}&\supseteq \bigcup_{j=n_k+1}^{n_{k+1}}  E_{j,\theta\gamma}F_j\\ \notag
&=\bigcup_{j=n_k+1}^{n_{k+1}}\left[ (E_{n_{k}+1, \theta\gamma}F_{n_{k}+1})^c\cdots (E_{j-1, \theta\gamma}F_{j-1})^c (E_{j, \theta\gamma}F_{j})\right]\\ \notag
&\supseteq \bigcup_{j=n_k+1}^{n_{k+1}}\left[ E_{n_{k}+1, \theta\gamma}^c\cdots E_{j-1, \theta\gamma}^c E_{j, \theta\gamma}F_{j}\right].\\ \notag
\end{align}

Now, let us estimate the probability $ \mathbb P \left(E_{n_{k}+1, \theta\gamma}^c\cdots E_{j-1, \theta\gamma}^c E_{j, \theta\gamma}F_{j}\right)$ by the Markov property and the  moderate deviations result in Theorem \ref{thm main}.

Let $\FF_t^{\xi}:=\sigma\{\xi_s; 0\le s\le t\}$. By the Markov property of Ornstein-Uhlenbeck process $\xi$, for any $ j\in(n_k, n_{k+1}]$, we have
\begin{align}\label{eq: I}
\PP(E_{n_k+1,\theta\gamma}^c\cdots E_{j-1,\theta\gamma}^c E_{j,\theta\gamma}F_j)&=\int_{E_{n_k+1,\theta\gamma}^c\cdots E_{j-1,\theta\gamma}^cE_{j,\theta\gamma}}\PP(F_j|\FF_j^{\xi})\dif \PP\\ \notag
&=\int_{E_{n_k+1,\theta\gamma}^c\cdots E_{j-1,\theta\gamma}^cE_{j,\theta\gamma}} \PP(F_j|\xi_j)\dif \PP \\ \notag
&\ge\int_{E_{n_k+1,\theta\gamma}^c\cdots E_{j-1,\theta\gamma}^cE_{j,\theta\gamma}}\PP(F_j|\xi_j)1_{[|\xi_j|<j^{\frac18 }]}\dif \PP.
\end{align}
When $|\xi_j|<j^{\frac18}$, for any $j\in[n_{k+1}-\gamma^{\frac{1+k}{4}}, n_{k+1}]$, by Chebychev's inequality and Lemma \ref{lem OU}, we have
\begin{align}\label{eq: I1}
\PP(F_j^c|\xi_j)=&\PP\left(|S_{n_{k+1}}-S_{j}|\ge   \sqrt{n_{k+1}\log\log k}\sigma \Big|\xi_j    \right)\\ \notag
\le & ( \sigma^2 n_{k+1}\log\log k)^{-1} \E\left(|S_{n_{k+1}}-S_{j}|^2\Big|\xi_j  \right)\\ \notag
\le & c(\sigma^2 n_{k+1}\log\log k)^{-1}(1+\gamma^{\frac{k+1}{2}}+j^{\frac12}\gamma^{\frac{k+1}{4}}).\\ \notag
 \end{align}
When $|\xi_j|<j^{\frac18}$,  $j\in(n_k, n_{k+1}-\gamma^{\frac{1+k}{4}}]$, by Chebychev's inequality and Lemma \ref{lem OU} again,
\begin{align}\label{eq: I2}
&\PP(F_j^c|\xi_j)=\PP\left(\left|S_{n_{k+1}}-S_{j}\right|\ge  \sqrt{n_{k+1}\log\log k}\sigma \Big|\xi_j    \right)\\ \notag
\le & \PP\left(\left|S_{ j+\gamma^{(k+1)/8}}-S_j\right|\ge \frac{\sigma }{2}\sqrt{n_{k+1}\log\log k}\Big|\xi_j    \right)+\PP\left(\left|S_{n_{k+1}}-S_{ j+\gamma^{(k+1)/8}} \right|\ge \frac{\sigma}{2} \sqrt{n_{k+1}\log\log k}\Big|\xi_j    \right)\\ \notag
\le & c( \sigma^2 n_{k+1}\log\log k)^{-1}(1+\gamma^{\frac{k+1}{4}}+j^{\frac12 }\gamma^{\frac{k+1}{8}})+\PP\left(\left|S_{n_{k+1}}-S_{ j+\gamma^{(k+1)/8}} \right|\ge \frac{ \sigma}{2}\sqrt{n_{k+1}\log\log k} \Big|\xi_j    \right).
\end{align}
 Next, we shall use the moderate deviations result to estimate the last term.  For any $|x|<j^{\frac 18},t=\gamma^{(k+1)/8}$, the transition probability $P(t; x,\cdot)$ is absolutely continuous with respect to the invariant measure $\mu$, and the Radon-Nikodym derivative
 \begin{align*}
 f:=\frac{\dif P(t;x, \cdot)}{\dif \mu}\in L^2(\mu) \text{ and } \|f\|_{L^2(\mu)} \le l \ \ \  \text{for some } l>0.
  \end{align*}For any $|\xi_j|<j^{\frac{1 }{8}}$, applying the MDP in Theorem \ref{thm main}, %to $\lambda(n_{k+1}-j-\gamma^{(k+1)/8} )=\sqrt{n_{k+1}\log k}/ \sqrt{n_{k+1}-j-\gamma^{(k+1)/8}}$ for  $
   we obtain that for any $j\in [n_k, n_{k+1}-\gamma^{\frac{k+1}{8}}]$ large enough
 \begin{align} \label{eq: I3}
 &\PP\left(\left|S_{n_{k+1}}-S_{ j+\gamma^{(k+1)/8}} \right|\ge   \sqrt{n_{k+1}\log\log k}\sigma/2 \Big|\xi_j    \right)\\ \notag
 =&\E\left[\PP_{\xi_{j+\gamma^{(k+1)/8}}}\left(\left|S_{n_{k+1}- j-\gamma^{(k+1)/8}} \right|\ge  \sqrt{n_{k+1}\log \log k}\sigma/2\right)\Big|\xi_j\right]\\ \notag
 \le &  \exp\left\{-\frac{ n_{k+1}\log\log k}{16(n_{k+1}-j-\gamma^{(k+1)/8})}  \right\}.
   \end{align}
By \eqref{eq: I1}-\eqref{eq: I3}, for any $|\xi_j|<j^{\frac{1 }{8}}$ and for any  $j$ large enough,
$$\PP(F_j^c|\xi_j)\le \frac12.$$
 This, together with \eqref{eq: I}, implies that for any $k$ large enough,
\begin{align*}
\PP(E_{n_k+1,\theta\gamma}^c\cdots E_{j-1,\theta\gamma}^c E_{j,\theta\gamma}F_j)
&\ge\frac12\PP\left( E_{n_k+1,\theta\gamma}^c\cdots E_{j-1,\theta\gamma}^cE_{j,\theta\gamma}\cap\{|\xi_j|<j^{\frac18 }\}\right).
\end{align*}
Then
\begin{align*}
 &\PP(E_{n_k+1,\theta\gamma}^c\cdots E_{j-1,\theta\gamma}^c E_{j,\theta\gamma})\\ \notag
 \le&\PP\left( E_{n_k+1,\theta\gamma}^c\cdots E_{j-1,\theta\gamma}^cE_{j,\theta\gamma}\cap\{|\xi_j|<j^{\frac18 }\}\right)+\PP\left(|\xi_j|\ge j^{\frac{1}{8}} \right)\\ \notag
 \le &  2\PP(E_{n_k+1,\theta\gamma}^c\cdots E_{j-1,\theta\gamma}^c E_{j,\theta\gamma}F_j)+\PP\left(|\xi_j|\ge j^{\frac{1}{8}} \right).
\end{align*}
 First taking sum of $j$ over $(n_k+1, n_{k+1}]$, and then taking sum of $k$ over $\mathbb N$,  by \eqref{eq: I0}, we have
 $$
 \sum_{k\ge1}\PP\left(\bigcup_{j=n_{k}+1}^{n_{k+1}}E_{j,\theta\gamma} \right)\le 2\sum_{k\ge1}\PP\left(  E_{n_{k+1},\theta }\right)+\sum_{k\ge1}\sum_{j=n_{k}+1}^{n_{k+1}}\PP(|\xi_j|\ge j^{\frac{1}{8}}), $$
 which is finite by \eqref{eq: finite} and the fact that the law of $\xi_n$ is identical to the invariant Gaussian measure $\mu$. Consequently by the Borel-Cantelli lemma, we have
 \begin{align*}
 \PP\left(\bigcup_{j=n_{k}+1}^{n_{k+1}}E_{j,\theta\gamma} \text{ i.o.} \right)=0,
 \end{align*}
 which is equivalent to that
  \begin{align*}
 \PP\left( E_{n,\theta\gamma} \text{ i.o.} \right)=0.
 \end{align*}
 By the arbitrariness of $\theta, \gamma\in(1,+\infty)$, we have
 \begin{align*}
\limsup_{n\rightarrow +\infty}\frac{S_{n}}{\sqrt{2n\log\log n}}\le \sigma\ \ \ a.s..
\end{align*}

 \vskip0.3cm

    {\it Step 1.3.  Comparison between the sequence $\{S_t\}_{t\ge1, t\in\R}$  and the  sequence $\{S_{n}\}_{n\ge1,n\in \N}$.}
For any $\e>0, n\in\N$, by Lemma \ref{lem S} and Chebyshev's inequality, we have
 \begin{align*}
&\PP\left(\sup_{n<t\le n+1}\frac{|S_t-S_n|}{\sqrt{2n\log\log n}} >\e\right)\\ \notag
\le &(2\e^2n\log\log n)^{-2}   \E\left[\sup_{n<t\le n+1} |S_t-S_n|\right]^4\\ \notag
\le &  (2\e^2 n\log\log n)^{-2} c_4.
 \end{align*}
Consequently, we have
$$
\sum_{n\in \N}\PP\left(\sup_{n<t\le n+1}\frac{|S_t-S_n|}{\sqrt{2n\log\log n}} >\e\right)<+\infty.
$$
By the Borel-Cantelli lemma, we obtain
\begin{align*}
\limsup_{n\rightarrow +\infty} \sup_{n<t\le n+1}\frac{|S_t-S_n|}{\sqrt{2n\log\log n}} \le\e  \ \ a.s..
\end{align*}
Due to the arbitrariness of $\e>0$,
$$
\limsup_{\R\ni t\rightarrow +\infty}\frac{S_t}{\sqrt{2t\log\log t}} \le
\limsup_{\N \ni n\rightarrow +\infty}\frac{S_n}{\sqrt{2n\log\log n}} \ \ \ \  \ a.s.,
$$
which implies the upper bound \eqref{upper} holds.

\vskip0.5cm

Step 2. ({\bf  Lower bound}). In this step, we shall prove that:   the sequence  $\{S_{n_k}\}_{k\ge1}$, where $n_k=\rho^k$ for any given $\rho>4$, satisfies that
\begin{align}\label{lower}
\limsup_{ k\rightarrow +\infty}\frac{S_{n_k}}{\sqrt{2n_k\log\log n_k}} \ge (1-1/\rho-1/\sqrt \rho)\sigma, \ \  a.s..
\end{align}
By the arbitrariness of $\rho>4$, we have
$$
\limsup_{ n\rightarrow +\infty}\frac{S_{n}}{\sqrt{2n\log\log n}} \ge \sigma\ \ \ a.s.
$$

For any $\theta\in(0,1)$,  let
$$
G_{k, \theta}:=\{S_{n_{k}}-S_{n_{k-1}+k}\ge \theta\sigma\sqrt{2n_k\log\log n_k}\}
$$
and
$$
H_{k, \theta}:=\{|S_{n_{k-1}+k}|\ge \theta \sigma\sqrt{2n_k\log\log n_k} \}.
$$
Recall  $E_{n, \theta}$ defined in  \eqref{E}. For any $k\ge1$, $1/\sqrt \rho<\theta_2\le \theta_1<1-1/\rho$, we have
\begin{align*}
G_{k, \theta_1}\subset&\left\{ S_{n_{k}}-S_{n_{k-1}+k}\ge \theta_1\sigma\sqrt{2n_k\log\log n_k}, |S_{n_{k-1}+k}|< \theta_2 \sigma\sqrt{2n_k\log\log n_k}  \right\}\\
&\cup\left\{|S_{n_{k-1}+k}|\ge \theta_2 \sigma\sqrt{2n_k\log\log n_k}  \right\}\\
\subset& E_{n_k, \theta_1-\theta_2}\cup H_{k, \theta_2}.
\end{align*}
Taking the union of the terms from $k$ to $\infty$, we have
\begin{align*}
\bigcup_{j=k}^{\infty}G_{j, \theta_1}\subset  \left(\bigcup_{j=k}^{\infty}E_{n_j, \theta_1-\theta_2}\right) \bigcup \left(\bigcup_{j=k}^{\infty} H_{j, \theta_2}\right).
\end{align*}
Then
\begin{align}\label{eq: III}
\PP\left(\bigcup_{j=k}^{\infty}G_{j, \theta_1}\right)\le  \PP\left(\bigcup_{j=k}^{\infty}E_{n_j, \theta_1-\theta_2} \right)+\sum_{j=k}^{\infty}\PP\left(H_{j, \theta_2} \right).
\end{align}
Now we estimate those probabilities.

  {\it Step 2.1.  Estimate of $\PP\left(\bigcup_{j=k}^{\infty}G_{j, \theta_1}\right)$}.   For any $m> k$, by the Markov property of $\xi$, we have
\begin{align}\label{eq: II}
\PP\left(\bigcap _{j=k}^m  G_{j, \theta_1}^c\right)=&\E\left[ \PP\left(\bigcap _{j=k}^m  G_{j, \theta_1}^c|\FF^{\xi}_{n_{m-1}}\right) \right]=\int_{\bigcap _{j=k}^{m-1}  G_{j, \theta_1}^c}\PP\left(   G_{m, \theta_1}^c|\xi_{n_{m-1}}\right)\dif \PP\\ \notag
\le &\PP\left(|\xi_{n_{m-1}}|>m \right)+\int_{\bigcap _{j=k}^{m-1}  G_{j, \theta_1}^c}\PP\left(   G_{m, \theta_1}^c|\xi_{n_{m-1}}\right)1_{[|\xi_{n_{m-1}}|\le m]}\dif \PP.\\ \notag
\end{align}
Since the law of  stationary process $\xi_n$ is identical to the Gaussian measure $\mu$ defined in \eqref{mu}, there exist constants $c_1>0, c_2>0$ such that
\begin{align}\label{eq: II1}
\PP\left(|\xi_{n_{m-1}}|>m \right)\le c_1\exp\{-c_2m^2\} \ \ \ \text{for all } m\in \N.
\end{align}

 Next, we shall use the moderate deviations result to estimate the last term in \eqref{eq: II}.  For any  $|x|\le m$, the transition probability $P(m;x, \cdot)$ is absolutely continuous with respect to the invariant measure $\mu$, and the Radon-Nikodym derivative
 \begin{align*}
 f:=\frac{\dif P(m;x,\cdot)}{\dif \mu}\in L^2(\mu) \text{ and } \|f\|_{L^2(\mu)} \le l \ \ \  \text{for some } l>0.
  \end{align*}
 For any $|\xi_{n_{m-1}}|\le m$, applying the MDP in Theorem \ref{thm main},  we obtain that for any $m$ large enough
 \begin{align*}
 &\PP\left(   G_{m, \theta_1} \big|\xi_{n_{m-1}}\right)\\ \notag
 =&\PP\left(S_{n_{m }}-S_{n_{m-1}+m}\ge \theta_1\sigma\sqrt{2n_m\log\log n_m}  \big|\xi_{n_{m-1}}    \right)\\ \notag
 =&\E\left[\PP_{\xi_{n_{m-1}+m  }}\left(S_{n_{m}-n_{m-1}-m}\ge \theta_1\sigma\sqrt{2n_m\log\log n_m}\right)\big|\xi_{n_{m-1}} \right]\\ \notag
 \ge &  \exp\left\{ -\frac{\theta_1 n_m\log\log n_m}{n_{m}-n_{m-1}-m}  \right\}.\\ \notag
  \end{align*}
For any $m$ large enough, $\log\log n_m=\log\log \rho^m\sim \log m$, then the above inequality implies that
\begin{align} \label{eq: II2}
\PP\left(   G_{m, \theta_1}^c |\xi_{n_{m-1}}\right)\le 1- m^{-\frac{\theta_1}{1-1/\rho}}\ \ \ \ \text{for all } |\xi_{n_{m-1}}|\le m .
 \end{align}
Putting \eqref{eq: II}-\eqref{eq: II2} together and iterating this procedure,  we obtain that for any $m>k$ large enough,
\begin{align*}
\PP\left(\bigcap _{j=k}^m  G_{j, \theta_1}^c\right)\le \prod_{j=k}^m( 1-  j^{-\frac{\theta_1}{1-1/\rho}})
+\sum_{j=k}^m c_1\exp\{-c_2j^2\}.
\end{align*}
 Since $\theta_1/(1-1/\rho)<1$, we have
\begin{align*}
\PP\left(\bigcup _{j=k}^{\infty}  G_{j, \theta_1}\right)\ge& 1-\prod_{j=k}^{\infty}( 1-  j^{-\frac{\theta_1}{1-1/\rho}})
-\sum_{j=k}^{\infty} c_1\exp\{-c_2j^2\}\\
=&1-\sum_{j=k}^{\infty} c_1\exp\{-c_2j^2\}.
\end{align*}
For any $n\ge1$,
\begin{align}\label{eq: III1}
\PP\left(\bigcup _{j=n}^{\infty}  G_{j, \theta_1}\right)\ge
\lim_{k\rightarrow \infty}\PP\left(\bigcup _{j=k}^{\infty}  G_{j, \theta_1}\right)
\ge \lim_{k\rightarrow\infty}\left(1- \sum_{j=k}^{\infty} c_1\exp\{-c_2j^2\}\right)
=1.
\end{align}
\vskip0.3cm
  {\it Step 2.2. Estimate of $\PP\left(H_{j, \theta_2} \right)$}. Applying the MDP in Theorem \ref{thm main} to $\lambda(n)= \log\log n$, for any given $\e\in (0, \theta^2-1/\rho)$, it holds that for any $k$ large enough
\begin{align*}
\PP\left(H_{k, \theta_2} \right)=&\PP\left( |S_{n_{k-1}+k}|\ge \theta_2 \sigma\sqrt{2n_k\log\log n_k}   \right)\\
\le& \exp\left\{ -\frac{(\theta_2^2-\e)n_k\log\log n_k}{(n_{k-1}+k)}     \right\}\\
\sim & \exp\left\{ -(\theta_2^2-\e) \rho \log k \right\} \ \ \text{for large  } k,
\end{align*}
where we have used the facts   $\log\log n_k=\log\log \rho^k\sim \log k$ and $n_k/(n_{k-1}+k)\sim \rho$ for large $k$.
Since $(\theta^2-\e)\rho>1$,
\begin{align}\label{eq: III2}
\sum_{j=k}^{\infty}\PP\left(H_{k, \theta_2} \right)\le \sum_{j=k}^{\infty} j^{-(\theta_2^2-\e) \rho}\rightarrow 0,\ \ \ \text{as } k\rightarrow +\infty.
\end{align}
Putting \eqref{eq: III}, \eqref{eq: III1} and \eqref{eq: III2} together,  we have
\begin{align*}
\PP\left(\bigcup_{j=n}^{\infty}E_{n_j, \theta_1-\theta_2} \right)
\ge &\lim_{k\rightarrow\infty}
\PP\left(\bigcup_{j=k}^{\infty}E_{n_j, \theta_1-\theta_2} \right)\\
\ge&\lim_{k\rightarrow \infty}\left( \PP\left(\bigcup_{j=k}^{\infty}G_{j, \theta_1}\right)  -\sum_{j=k}^{\infty}\PP\left(H_{j, \theta_2} \right)\right)\\
=&1.
\end{align*}
The above inequlities implies that
$$
\PP(E_{n_k, \theta_1-\theta_2} \ \ \text{i.o.})=1,
$$
which implies that
\begin{align}
\limsup_{ k\rightarrow +\infty}\frac{S_{n_k}}{\sqrt{2n_k\log\log n_k}} \ge (\theta_1-\theta_2)\sigma, \ \  a.s..
\end{align}
By the arbitrariness of $1/\sqrt\rho<\theta_2<\theta_1<1-1/\rho$,  the lower bound \eqref{lower} holds. The proof is complete.
\end{proof}

\vskip0.5cm

\vskip0.5cm

\noindent{\bf Acknowledgments}:
R. Wang thanks the  Faculty of Science and Technology, University of Macau, for inviting and supporting.    He is supported by Natural Science Foundation of China 11301498, 11431014 and the Fundamental Research Funds for the Central Universities.   L. Xu is supported by the grant SRG2013-00064-FST. Both of the   authors are supported by the research project RDAO/RTO/0386-SF/2014.

\bibliographystyle{amsplain}

\end{document}